\RequirePackage{etex}
\pdfoutput=1
\documentclass{amsart}
\usepackage{mathtools}
\usepackage{amssymb}
\usepackage{xparse}
\usepackage{xr-hyper}
\usepackage{array}

\usepackage{subdepth}
\usepackage{bm,bbm}
\usepackage{mathrsfs}
\usepackage{stmaryrd}
\usepackage[scr=boondoxo]{mathalfa}

\usepackage[centering, footskip=6mm]{geometry}

\usepackage{xpatch}
\usepackage{mdframed}

\makeatletter
\xpatchcmd{\endmdframed}
  {\aftergroup\endmdf@trivlist\color@endgroup}
  {\endmdf@trivlist\color@endgroup\@doendpe}
  {}{}
\makeatother

\usepackage{iftex}

\ifpdf
\usepackage[hypertexnames=false]{hyperref}
\else
\usepackage[dvipdfmx,hypertexnames=false]{hyperref}
\fi

\usepackage[nameinlink]{cleveref}

\usepackage{etoolbox}
\makeatletter
\patchcmd\@thm
  {\let\thm@indent\indent}{\let\thm@indent\noindent}%
  {}{}
\expandafter\patchcmd\csname\string\proof\endcsname
  {\normalparindent}{0pt }{}{}
\makeatother

\makeatletter
\renewcommand\th@plain{\slshape}
\makeatother
\newtheoremstyle{plain}
  {-\topsep}
  {}
  {\slshape}
  {}
  {\sffamily\bfseries}
  {.}
  {.5em}
  {}
\theoremstyle{plain}
\newtheorem{theorem}{Theorem}

\newtheorem{lemma}[theorem]{Lemma}

\newtheorem*{claim*}{Claim}

\surroundwithmdframed[skipabove=\medskipamount,skipbelow=\medskipamount]{theorem}
\surroundwithmdframed[skipabove=\medskipamount,skipbelow=\medskipamount]{corollary}
\surroundwithmdframed[skipabove=\medskipamount,skipbelow=\medskipamount]{lemma}
\surroundwithmdframed[skipabove=\medskipamount,skipbelow=\medskipamount]{proposition}
\surroundwithmdframed[skipabove=\medskipamount,skipbelow=\medskipamount]{claim}
\surroundwithmdframed[skipabove=\medskipamount,skipbelow=\medskipamount]{claim*}

\newtheoremstyle{definition}
  {-\topsep}
  {}
  {\normalfont}
  {}
  {\sffamily\bfseries}
  {.}
  {.5em}
  {}
\theoremstyle{definition}

\newtheorem{remark}[theorem]{Remark}

\surroundwithmdframed[skipabove=\medskipamount,skipbelow=\medskipamount]{exercise}
\surroundwithmdframed[skipabove=\medskipamount,skipbelow=\medskipamount]{definition}
\surroundwithmdframed[skipabove=\medskipamount,skipbelow=\medskipamount]{notation}
\surroundwithmdframed[skipabove=\medskipamount,skipbelow=\medskipamount]{problem}
\surroundwithmdframed[skipabove=\medskipamount,skipbelow=\medskipamount]{question}
\surroundwithmdframed[skipabove=\medskipamount,skipbelow=\medskipamount]{note}
\surroundwithmdframed[skipabove=\medskipamount,skipbelow=\medskipamount]{remark}
\surroundwithmdframed[skipabove=\medskipamount,skipbelow=\medskipamount]{example}

\crefname{section}{Section}{Sections}

\crefname{theorem}{Theorem}{Theorems}
\crefname{corollary}{Corollary}{Corollaries}
\crefname{lemma}{Lemma}{Lemmas}
\crefname{proposition}{Proposition}{Propositions}
\crefname{claim}{Claim}{Claims}
\crefname{definition}{Definition}{Definitions}
\crefname{notation}{Notation}{Notations}
\crefname{problem}{Problem}{Problems}
\crefname{question}{Question}{Questions}
\crefname{note}{Note}{Notes}
\crefname{remark}{Remark}{Remarks}
\crefname{example}{Example}{Examples}

\crefname{enumi}{}{}
\crefname{enumii}{}{}
\crefname{enumiii}{}{}

\crefformat{equation}{{\upshape(#2#1#3)}}

\usepackage{autonum}

\makeatletter
\newcommand{\restore@Environment}[1]{%
  \AtBeginDocument{%
    \csletcs{#1*}{#1}%
    \csletcs{end#1*}{end#1}%
  }%
}
\forcsvlist\restore@Environment{alignat,equation,gather,multline,flalign,align}
\makeatother

\usepackage{enumitem}
\setlist{leftmargin=20pt}
\setlist[enumerate]{label=\textup{(\roman*)}}

\ifpdf
\usepackage{graphicx}
\else
\usepackage[dvipdfmx]{graphicx}
\fi
\usepackage{tikz}
\usetikzlibrary{arrows.meta}
\usetikzlibrary{tikzmark}
\usetikzlibrary{cd}
\tikzcdset{ampersand replacement=\&}
\tikzcdset{
  cells={font=\everymath\expandafter{\the\everymath\displaystyle}},
}

\allowdisplaybreaks

\newif\ifendnotes

\endnotesfalse

\ifendnotes
\usepackage{endnotes}
\fi

\newif\ifshowkeys

\showkeysfalse

\ifshowkeys
\usepackage{showkeys}

\makeatletter
  \SK@def\cref#1{\SK@\SK@@ref{#1}\SK@cref{#1}}%
\makeatother
\fi

\let\tmp\phi
\let\phi\varphi
\let\varphi\tmp
\let\tmp\epsilon
\let\epsilon\varepsilon
\let\varepsilon\tmp

\renewcommand{\subset}{\subseteq}

\renewcommand{\mod}[1]{(\mathrm{mod}\ #1)}

\renewcommand{\and}{\quad\text{and}\quad}

\makeatletter
\NewDocumentCommand{\xsideset}{mmme{_^}}{%
\mathop{%
\settowidth{\dimen0}{$\m@th\displaystyle#3$}%
\dimen0=.5\dimen0
\settowidth{\dimen2}{$%
\m@th\displaystyle#3%
\IfValueT{#4}{_{#4}}%
\IfValueT{#5}{^{#5}}%
$}%
\dimen2=.5\dimen2
\advance\dimen2 -\dimen0
\sbox6{\scriptspace\z@$\displaystyle{\vphantom{#3}}#1$}
\sbox8{\scriptspace\z@$\displaystyle{\vphantom{#3}}#2$}
\ifdim\wd6>\dimen2 \kern\dimexpr\wd6-\dimen2\relax\fi
{%
\mathop{\llap{\copy6}{\displaystyle#3}\rlap{\copy8}}\limits%
\IfValueT{#4}{_{#4}}%
\IfValueT{#5}{^{#5}}%
}%
\ifdim\wd8>\dimen2 \kern\dimexpr\wd8-\dimen2\relax\fi
}%
}
\makeatother

\newcommand{\dsum}[1]{\xsideset{}{^{#1}}\sum}
\newcommand{\astsum}{\dsum{\smash{\ast}}}

\begin{document}

\title[Counting relatively prime pairs of palindromes]
{Counting relatively prime pairs of palindromes}
\author[H. Kobayashi]{Hirotaka Kobayashi}
\author[Y. Suzuki]{Yuta Suzuki}
\author[R. Umezawa]{Ryota Umezawa}
\maketitle

\begin{abstract}
For a given base $g\ge2$,
a positive integer is called a palindrome
if its base $g$ expansion reads the same backwards as forwards.
In this paper, we give an asymptotic formula
for the number of relatively prime pairs
of palindromes of a fixed odd length and of any base $g\ge2$,
which solves an open problem proposed by Banks and Shparlinski (2005).
\end{abstract}

\section{Introduction}
\label{sec:intro}
Take an integer $g\ge2$.
For a positive integer $n$,
we write its base $g$ expansion as
\begin{equation}
\label{base_g_expansion}
n
=
n_{N-1}\cdots n_{0}{}_{(g)}
\coloneqq
\sum_{i=0}^{N-1}n_{i}g^{i}
\quad\text{with}\quad
N\in\mathbb{N},\ 
n_{0},\ldots,n_{N-1}\in\{0,\ldots,g-1\}
\ \text{and}\ 
n_{N-1}\neq0.
\end{equation}
When $n\in\mathbb{N}$
has the base $g$ expansion \cref{base_g_expansion},
we say the integer $n$ or the expansion \cref{base_g_expansion} is of length $N$.
In this paper, we are interested in the palindromes of base $g$,
the positive integers having the base $g$ expansion
which reads the same backwards as forwards.
More precisely, we say a positive integer $n$ is a palindrome of base $g$
if the base $g$ expansion \cref{base_g_expansion} of $n$ satisfies
\begin{equation}
\label{def:palindrome:eq}
n_{i}=n_{N-i-1}
\quad\text{for}\quad
i\in\{0,\ldots,N-1\}.
\end{equation}
Since we always use $g$ as the base,
we just say ``palindromes'' without specifying the base.

Banks and Shparlinski~\cite[p.~101]{BanksShparlinski:PalindromePhi}
pointed out that any asymptotic formula for the number of relatively prime pairs
of palindromes of a fixed length is not known.
They formulated this problem with palindromes of even length
but it is inappropriate since every palindrome of even length $2N$ with $N\in\mathbb{N}$
is divisible by $g+1\ge3$ as one can see that if a palindrome $n$ of length $2N$
has the base $g$ expansion \cref{base_g_expansion} with $N$ replaced by $2N$, then
\begin{align}
n
=
\sum_{i=0}^{N-1}
n_{i}g^{i}
+
\sum_{i=N}^{2N-1}
n_{i}g^{i}
&=
\sum_{i=0}^{N-1}
n_{i}g^{i}
+
\sum_{i=0}^{N-1}
n_{2N-i-1}g^{2N-i-1}\\
&=
\sum_{i=0}^{N-1}
n_{i}g^{i}(1+g^{2(N-i)-1})
\equiv0\ \mod{g+1}
\end{align}
since $g\equiv-1\ \mod{g+1}$.
We thus consider this problem with palindrome of odd length $2N+1$ with $N\in\mathbb{N}$,
i.e.\ we are interested in obtaining some asymptotic formula for the sum
\begin{equation}
\label{main_sum}
\sum_{\substack{
m,n\in\Pi(2N+1)\\
(m,n)=1
}}
1,
\end{equation}
where
\begin{equation}
\label{def:Pi_N}
\Pi(N)
\coloneqq
\{
n:\text{palindrome of length $N$}
\}.
\end{equation}
Note that $\#\Pi(N)\asymp g^{\frac{N}{2}}$ as seen in \cref{lem:Pi_count} below,
where the implicit constant depends only on $g$.

For a given set $\mathscr{A}$ of integers, write
\begin{equation}
\label{def:A_aq}
\mathscr{A}(a,q)
\coloneqq
\{
n\in\mathscr{A}
\mid
n\equiv a\ \mod{q}
\},
\end{equation}
where if the set of integers is written with parameters as $\mathscr{A}(x,\ldots,y)$, we write instead
\begin{equation}
\label{def:A_xy_aq}
\mathscr{A}(x,\ldots,y;a,q)
\coloneqq
\{
n\in\mathscr{A}(x,\ldots,y)
\mid
n\equiv a\ \mod{q}
\}.
\end{equation}
At the time of the paper~\cite{BanksShparlinski:PalindromePhi},
the best admissible range of the asymptotic formula
\[
\#\Pi(N;a,q)
=
\frac{1}{q}\#\Pi(N)+(\text{error})
\quad\text{with}\quad
(q,g^{3}-g)=1
\]
was
\begin{equation}
\label{admissible_range:BHS}
q\le c\biggl(\frac{N}{\log N}\biggr)^{\frac{1}{2}}
\end{equation}
with some constant $c=c(g)>0$,
which is due to Banks, Hart and Sakata~\cite[Proposition~4.2]{BHS}.
This admissible range \cref{admissible_range:BHS}
was improved later by Col~\cite[Th\'eor\`eme~1]{Col:PalindromeAP} to
\begin{equation}
\label{admissible_range:Col}
q\le\exp\biggl(c\frac{N}{\log N}\biggr)
\end{equation}
with some constant $c=c(g)>0$.
Even with such a narrow admissible range,
Banks and Shparlinski succeeded in calculating the average of the Euler totient function
over palindromes of a fixed length by proving a weak analog of the Brun--Titchmarsh theorem
(see \cite[Theorem~7]{BanksShparlinski:PrimeDivisorPalindrome}
and \cite[Lemma~2]{BanksShparlinski:PalindromePhi})
for palindromes
\begin{equation}
\label{Banks_Shparlinski_BT}
\#\Pi(N;a,q)
\ll
\left\{
\begin{array}{>{\displaystyle}ll}
g^{\frac{N}{2}}q^{-\frac{1}{2}}+1&\text{if $a\not\equiv0\ \mod{q}$},\\[2mm]
g^{\frac{N}{2}}q^{-\frac{1}{2}}&\text{if $a\equiv0\ \mod{q}$},
\end{array}
\right.
\end{equation}
where the implicit constant depends only on $g$.
For the proof and an additional remark,
see \cref{lem:BS_Brun_Titchmarsh} below and its subsequent remark.
After the paper~\cite{BanksShparlinski:PalindromePhi},
there has been several progress on the distribution of palindromes
~\cite{Col:PalindromeAP,TuxanidyPanario}.
Most notably, the Bombieri--Vinogradov type result is now available for palindromes
with the level of distribution $\frac{1}{5}$, which is due to Tuxanidy and Panario~\cite{TuxanidyPanario}.
However, what Banks and Shparlinski pointed out as the difficulty in \cref{main_sum}
is the behavior of $\#\Pi(N;0,q)$ with large $q$,
i.e.\ the bound \cref{Banks_Shparlinski_BT} has only the $q$-exponent $-\frac{1}{2}$.
This produces the error of the size
\begin{equation}
\label{tail_bound}
\sum_{g^{\delta(2N+1)}<d<g^{2N+1}}
\mu(d)\#\Pi(2N+1;0,d)^{2}
\ll
\sum_{g^{\delta(2N+1)}<d<g^{2N+1}}
(g^{N}d^{-\frac{1}{2}})^{2}
\asymp_{\delta}
Ng^{2N}
\end{equation}
with some $\delta\in(0,1)$ in the usual argument using the M\"obius function,
which supersedes the expected main term.
Since palindromes are as sparse as squares,
this bad behavior in the distribution in arithmetic progressions
seems unavoidable.

In this paper, we show that 
it is indeed enough to combine
the Bombieri--Vinogradov type result due to Tuxanidy and Panario
(or even the weaker result due to Col~\cite[Th\'{e}or\`{e}me~2]{Col:PalindromeAP})
and \cref{Banks_Shparlinski_BT}
to obtain the following answer to the problem of Banks and Shparlinski:
\begin{theorem}
\label{thm:relatively_prime_palindrome}
For $N\ge1$, we have
\[
\sum_{\substack{
m,n\in\Pi(2N+1)\\
(m,n)=1
}}
1
=
\frac{\rho(g)}{\zeta(2)}
\biggl(
1-\frac{2}{g}\prod_{p\mid g}\biggl(\frac{p}{p+1}\biggr)
\biggr)
\#\Pi(2N+1)^{2}
\Bigl(1+O(g^{-c\sqrt{N}})\Bigr),
\]
where $c>0$ is some constant,
\begin{equation}
\label{thm:relatively_prime_palindrome:rho_def}
\rho(g)
\coloneqq
\left\{
\begin{array}{>{\displaystyle}cl}
\biggl(\frac{g}{g-1}\biggr)^{2}&\text{when $g$ is even},\\[4mm]
\frac{g+\frac{1}{3}}{g-1}&\text{when $g$ is odd}
\end{array}
\right.
\end{equation}
and the constant $c>0$ and the implicit constant depends only on $g$.
\end{theorem}

Without fixing the length,
the behavior of the number of relatively prime pairs of palindromes up to $x$
seems rather wild since it heavily depends on the size of $x$
with respect to the power of base $g$.
However, by considering only palindromes coprime to $g^{3}-g$,
we can obtain a clean asymptotic formula by using the same method as \cref{thm:relatively_prime_palindrome}.
Let us write
\begin{equation}
\label{def:P_ast}
\mathscr{P}^{\ast}(x)
\coloneqq
\{
n\in[1,x]\cap\mathbb{N}
\mid
\text{$n$ is palindrome and $(n,g^{3}-g)=1$}
\}.
\end{equation}
\begin{theorem}
\label{thm:relatively_prime_palindrome:g3_g}
For $N\ge1$, we have
\[
\sum_{\substack{
m,n\in\mathscr{P}^{\ast}(x)\\
(m,n)=1
}}
1
=
\frac{1}{\zeta(2)}
\prod_{p\mid g^{3}-g}
\biggl(1-\frac{1}{p^{2}}\biggr)^{-1}
\#\mathscr{P}^{\ast}(x)^{2}
\Bigl(1+O(e^{-c\sqrt{\log x}})\Bigr)
\]
with some constant $c>0$,
where $c$ and the implicit constant depend only on $g$.
\end{theorem}

The idea of the proof of \cref{thm:relatively_prime_palindrome}
is just a small trick in the bound \cref{tail_bound}.
We use \cref{Banks_Shparlinski_BT} to only one factor $\#\Pi(2N+1;0,d)$
and swap the summation in the remaining double sum as
\begin{equation}
\label{tail_bound_new}
\begin{aligned}
\sum_{g^{\delta(2N+1)}<d<g^{2N+1}}
\mu(d)\#\Pi(2N+1;0,d)^{2}
&\ll
g^{(1-\delta) N}
\sum_{g^{\delta(2N+1)}<d<g^{2N+1}}
\sum_{\substack{
n\in\Pi(2N+1)\\
d\mid n
}}
1\\
&\ll
g^{(1-\delta)N}
\sum_{n\in\Pi(2N+1)}
\tau(n).
\end{aligned}
\end{equation}
This enables us to remove the contribution of the sum over $d$
by using the well-known bound $\tau(n)\ll n^{\epsilon}$
of the divisor function.
The proof of \cref{thm:relatively_prime_palindrome:g3_g}
is completely parallel to \cref{thm:relatively_prime_palindrome}.

\section{Notation}
\label{sec:notation}
Throughout the paper,
$y$, $Q$, $R$, $U$, $\delta$ denote positive real numbers,
$x$, $\alpha$ denotes a real number,
$d$, $f$, $i$, $m$, $n$, $q$, $r$, $M$, $N$ denote non-negative integers and
$a$, $h$, $k$, $\ell$ denote integers.
The letter $p$ is reserved for prime numbers.
The letter $\epsilon$ denotes a positive real number
smaller than some absolute constant.
The letter $c$ is used for positive constants
which can take different values line by line.

For a given set $\mathscr{A}$ of integers,
$\#\mathscr{A}$ denotes the cardinality of $\mathscr{A}$
and we use the notation $\mathscr{A}(a,q)$ and $\mathscr{A}(x,\ldots,y;a,q)$
given in \cref{def:A_aq} and \cref{def:A_xy_aq}.

For a real number $x$, we let $e(x)\coloneqq\exp(2\pi ix)$.

For a positive integer $q$, the symbol
\[
\astsum_{a\ \mod{q}}
\]
stands for the sum over reduced residues $\mod{q}$.

Throughout the paper, $g\ge2$ is an integer used as the base.
For $N\in\mathbb{N}$ and $k\in\mathbb{Z}$,
we use $\Pi(N)$ defined in \cref{def:Pi_N} and write
\begin{equation}
\label{def:Pi_k_g}
\Pi_{k,g}(N)
\coloneqq
\{
n\in\Pi(N)
\mid
n\equiv k\ \mod{g^{3}-g}
\}.
\end{equation}
For $N\in\mathbb{N}$ and $\alpha\in\mathbb{R}$, let us define
\begin{equation}
\label{def:S_N_alpha}
S(N;\alpha)
\coloneqq
\sum_{n\in\Pi(N)}e(\alpha n)
\end{equation}
and
\begin{equation}
\label{def:Phi_psi}
\Phi_{N}(\alpha)
\coloneqq
\prod_{1\le i<N}|\psi(\alpha(g^{i}+g^{2N-i}))|
\quad\text{with}\quad
\psi(\alpha)
\coloneqq
\sum_{0\le n<g}e(\alpha n).
\end{equation}
Also, we use $\mathscr{P}^{\ast}(x)$ defined by \cref{def:P_ast}.

The arithmetic functions $\tau(n),\mu(n)$ stand for
the divisor function (the number of divisors of $n$) and the M\"obius function, respectively.
The function $\zeta(s)$ stands for the Riemann zeta function.

For integers $m$ and $n$,
we write $(m,n)$ for the greatest common divisor of $m$ and $n$,
which is easily distinguished from the pair $(m,n)$ from the context.

For a logical formula $P$,
we write $\mathbbm{1}_{P}$
for the indicator function of $P$.

If Theorem or Lemma is stated
with the phrase ``where the implicit constant depends on $a,b,c,\ldots$'',
then every implicit constant in the corresponding proof
may also depend on $a,b,c,\ldots$ even without special mentions.

\section{Palindromes in arithmetic progressions}
\label{sec:Palindrome_AP}
We first count the palindromes of a fixed length:
\begin{lemma}
\label{lem:Pi_count}
For $N\in\mathbb{N}$, we have
\[
\#\Pi(2N)
=
g^{N-1}(g-1)
\and
\#\Pi(2N+1)
=
g^{N}(g-1).
\]
\end{lemma}
\begin{proof}
Since $n\in\Pi(2N)$ can be written as
\[
n
=
\sum_{i=0}^{N-1}
n_{i}(g^{i}+g^{2N-i-1})
\quad\text{with}\quad
n_{0},\ldots,n_{N-1}\in\{0,\ldots,g-1\}
\ \text{and}\ 
n_{0}\neq0,
\]
we have
\begin{align}
\#\Pi(2N)
=
\biggl(\sum_{1\le n_{0}<g}1\biggr)
\prod_{i=1}^{N-1}
\biggl(\sum_{0\le n_{i}<g}1\biggr)
=
g^{N-1}(g-1).
\end{align}
This completes the proof for $\#\Pi(2N)$.
Similarly, since $n\in\Pi(2N+1)$ can be written as
\[
n
=
\sum_{i=0}^{N-1}
n_{i}(g^{i}+g^{2N-i})
+
n_{N}g^{N}
\quad\text{with}\quad
n_{0},\ldots,n_{N}\in\{0,\ldots,g-1\}
\ \text{and}\ 
n_{0}\neq0,
\]
we have
\begin{align}
\#\Pi(2N+1)
=
\biggl(\sum_{1\le n_{0}<g}1\biggr)
\prod_{i=1}^{N}
\biggl(\sum_{0\le n_{i}<g}1\biggr)
=
g^{N}(g-1).
\end{align}
This completes the proof for $\#\Pi(2N+1)$.
\end{proof}

We then recall the weak Brun--Titchmarsh type result of Banks and Shparlinski:
\begin{lemma}[Banks--Shparlinski~\cite{BanksShparlinski:PrimeDivisorPalindrome}]
\label{lem:BS_Brun_Titchmarsh}
For $N,q\in\mathbb{N}$ and $a\in\mathbb{Z}$, we have
\[
\#\Pi(N;a,q)
\ll
\left\{
\begin{array}{>{\displaystyle}ll}
g^{\frac{N}{2}}q^{-\frac{1}{2}}+1&\text{if $a\not\equiv0\ \mod{q}$},\\[2mm]
g^{\frac{N}{2}}q^{-\frac{1}{2}}&\text{if $a\equiv0\ \mod{q}$},
\end{array}
\right.
\]
where the implicit constant depends only on $g$.
\end{lemma}
\begin{proof}
For $q\ge g^{N}$, the assertion is obvious since then
there is no multiple of $q$ in $[1,g^{N})$
and at most one integer $\equiv a\ \mod{q}$ in $[1,g^{N})$.
Also, by \cref{lem:Pi_count},
the assertion is obvious if $q<g$.
We may thus assume $g\le q<g^{N}$ without loss of generality.

Take the largest $r\in\mathbb{Z}_{\ge0}$ such that
\[
g^{r}\le q
\and
r\equiv N\ \mod{2},
\]
which is possible since we are assuming $g^{0}<g^{1}\le q$.
Since we are assuming $q<g^{N}$, we then have
\[
r<N
\and
g^{r}\asymp q.
\]
By $r\equiv N\ \mod{2}$,
any $n\in\Pi(N;a,q)$ can be decomposed as
\[
n=
{\underbrace{n_{N-1}\cdots n_{\frac{N+r}{2}}}_{m_{2}}
\underbrace{n_{\frac{N+r}{2}-1}\cdots n_{\frac{N-r}{2}}}_{m_{1}}
\underbrace{n_{\frac{N-r}{2}-1}\cdots n_{0}}_{m_{0}}}{}_{(g)},
\]
or, more precisely, $n$ can be written uniquely as
\[
n
=
m_{2}g^{\frac{N+r}{2}}
+m_{1}g^{\frac{N-r}{2}}
+m_{0}
\]
with
\[
m_{0},m_{2}\in[1,g^{\frac{N-r}{2}})\cap\mathbb{Z},\quad
m_{1}\in[0,g^{r})\cap\mathbb{Z}
\and
m_{2}g^{\frac{N-r}{2}}+m_{0}:\text{palindrome}.
\]
Let $d\coloneqq(q,g^{\frac{N-r}{2}})$. We then have $d\mid g^{\frac{N-r}{2}}\mid g^{\frac{N+r}{2}}$ and so
\[
n\equiv a\ \mod{q}
\implies
m_{0}\equiv a\ \mod{d}
\]
and so there are $\le g^{\frac{N-r}{2}}d^{-1}$ possibilities for $m_{0}$
since $d\mid g^{\frac{N-r}{2}}$.
For a given $m_{0}$,
the value of $m_{2}$ is uniquely determined
since $m_{2}$ should be the digital reverse of $m_{0}$.
For such $m_{0},m_{2}$, we have
\[
n\equiv a\ \mod{q}
\implies
m_{1}
\equiv
\overline{g^{\frac{N-r}{2}}/d}\cdot\frac{a-m_{2}g^{\frac{N+r}{2}}-m_{0}}{d}\ \mod{q/d},
\]
where $\overline{x}\ \mod{q/d}$ denotes the multiplicative inverse of $x\ \mod{q/d}$.
Since $g^{r}\le q$, there are at most $d$ possibilities for the value of $m_{1}$.
In total, we get
\[
\#\Pi(N;a,q)
\ll
g^{\frac{N-r}{2}}d^{-1}\cdot d
=
g^{\frac{N}{2}}g^{-\frac{r}{2}}
\ll
g^{\frac{N}{2}}q^{-\frac{1}{2}}.
\]
This completes the proof.
\end{proof}

\begin{remark}
\label{rem:BanksShparlinski_BS}
Note that it seems that the original bound
\begin{equation}
\label{BanksShparlinski_BS_incorrect}
\#\Pi(N;a,q)
\overset{\text{incorrect}}{\ll}
g^{\frac{N}{2}}q^{-\frac{1}{2}}
\end{equation}
stated as Lemma~2 in \cite[p.~96]{BanksShparlinski:PalindromePhi} is incorrect.
Indeed, with taking $a\in\Pi(N)$, e.g.\ $a=g^{N-1}+1$,
the original bound \cref{BanksShparlinski_BS_incorrect} implies
\[
1
\le
\#\Pi(N;a,q)
\ll
g^{\frac{N}{2}}q^{-\frac{1}{2}},
\]
which is false as $q\to\infty$.
This flaw is probably caused by ignorance of the case $q\ge g^{N}$.
\end{remark}

We next recall the result of Tuxanidy and Panario
on the level of distribution of palindromes
with some minor modifications.
Recall the definitions of $S(N;\alpha),\Phi_{N}(\alpha),\psi(\alpha)$
given in \cref{def:S_N_alpha} and \cref{def:Phi_psi}.
\begin{lemma}
\label{lem:S_to_Phi}
For $N\in\mathbb{N}$ and $\alpha\in\mathbb{R}$, we have
\[
S(2N+1;\alpha)
\ll
\Phi_{N}(\alpha),
\]
where the implicit constant depends only on $g$.
\end{lemma}
\begin{proof}
By writing $n\in\Pi(2N+1)$ as
\[
n
=
n_{0}(1+g^{2N})
+\sum_{i=1}^{N-1}n_{i}(g^{i}+g^{2N-i})
+n_{N}g^{N}
\quad\text{with}\quad
n_{0},\ldots,n_{N}\in\{0,\ldots,g-1\}
\ \text{and}\ 
n_{0}\neq0,
\]
we have
\begin{align}
&S(2N+1;\alpha)\\
&=
\sum_{1\le n_{0}<g}
\sum_{0\le n_{1},\ldots,n_{N-1}<g}
\sum_{0\le n_{N}<g}
e\biggl(
\alpha\biggl(
n_{0}(1+g^{2N})
+\sum_{i=1}^{N-1}n_{i}(g^{i}+g^{2N-i})
+n_{N}g^{N}
\biggr)\biggr)\\
&=
\biggl(\sum_{1\le n_{0}<g}e(\alpha n_{0}(1+g^{2N}))\biggr)
\biggl(\sum_{0\le n_{N}<g}e(\alpha n_{N}g^{N})\biggr)
\prod_{i=1}^{N-1}
\biggl(\sum_{0\le n_{i}<g}e(\alpha n_{i}(g^{i}+g^{2N-i}))\biggr).
\end{align}
By using the trivial bounds
\[
\biggl|\sum_{1\le n_{0}<g}e(\alpha n_{0}(1+g^{2N}))\biggr|,\quad
\biggl|\sum_{0\le n_{N}<g}e(\alpha n_{N}g^{N})\biggr|
\le
g,
\]
we obtain the lemma.
\end{proof}

\begin{lemma}[{Tuxanidy--Panario~\cite{TuxanidyPanario}}]
\label{lem:TuxanidyPanario_hybrid}
For $N\in\mathbb{N}$, $k\in\mathbb{Z}$, $Q\ge1$, $\epsilon>0$ and $\delta\in[\frac{1}{3},\frac{2}{5}-\epsilon]$,
we have
\[
\sum_{\substack{
2\le q\le Q\\
(q,g^{3}-g)=1
}}
\astsum_{h\ \mod{q}}
\Phi_{N}\biggl(\frac{h}{q}+\frac{k}{g^{3}-g}\biggr)
\ll
(
Q^{2}g^{(1-\delta-\sigma_{1}(g,\epsilon))N}
+
Q^{1-\frac{\sigma_{1}(g,\epsilon)}{\delta}}
g^{N}
)
e^{-\frac{\sigma_{\infty}(g)N}{\log Q}}
\]
with some constants $\sigma_{1}(g,\epsilon),\sigma_{\infty}(g)>0$,
where the implicit constant depends only on $g$ and $\epsilon$.
\end{lemma}
\begin{proof}
See Proposition~8.1 of \cite{TuxanidyPanario}.
\end{proof}

Recall the notation $\Pi_{k,g}(N)$ given in \cref{def:Pi_k_g}.
We prepare a minor variant of the Bombieri--Vinogradov type result of Tuxanidy--Panario:
\begin{lemma}
\label{lem:Tuxanidy_Panario_BV_variant}
For $N\in\mathbb{Z}_{\ge0}$, $\epsilon>0$ and $1\le Q\le g^{(\frac{4}{5}-\epsilon)N}$, we have
\[
\sum_{\substack{
1\le q\le Q\\
(q,g^{3}-g)=1
}}
q^{-\frac{1}{2}}
\max_{a,k\in\mathbb{Z}}
\biggl|
\#\Pi_{k,g}(2N+1;a,q)
-
\frac{1}{q}\#\Pi_{k,g}(2N+1)
\biggr|
\ll
g^{N-c\sqrt{N}}
\]
with some constant $c=c(g,\epsilon)>0$,
where the implicit constant depends only on $g$ and $\epsilon$.
\end{lemma}
\begin{proof}
By the orthogonality of additive characters, we have
\begin{align}
&\#\Pi_{k,g}(2N+1;a,q)
-
\frac{1}{q}\#\Pi_{k,g}(2N+1)\\
&=
\frac{1}{(g^{3}-g)q}
\sum_{\ell\ \mod{g^{3}-g}}
\sum_{\substack{
h\ \mod{q}\\
h\not\equiv0\ \mod{q}
}}
e\biggl(-\frac{ah}{q}-\frac{k\ell}{g^{3}-g}\biggr)
S\biggl(2N+1;\frac{h}{q}+\frac{\ell}{g^{3}-g}\biggr).
\end{align}
We thus have
\begin{align}
&\sum_{\substack{
1\le q\le Q\\
(q,g^{3}-g)=1
}}
q^{-\frac{1}{2}}
\max_{a,k\in\mathbb{Z}}
\biggl|
\#\Pi_{k,g}(2N+1;a,q)
-
\frac{1}{q}\#\Pi_{k,g}(2N+1)
\biggr|\\
&\le
\frac{1}{g^{3}-g}
\sum_{\ell\ \mod{g^{3}-g}}
\sum_{\substack{
2\le q\le Q\\
(q,g^{3}-g)=1
}}
q^{-\frac{3}{2}}
\sum_{\substack{
h\ \mod{q}\\
h\not\equiv0\ \mod{q}
}}
\biggl|S\biggl(2N+1;\frac{h}{q}+\frac{\ell}{g^{3}-g}\biggr)\biggr|\\
&\le
\max_{\ell\in\mathbb{Z}}
\sum_{\substack{
2\le q\le Q\\
(q,g^{3}-g)=1
}}
q^{-\frac{3}{2}}
\sum_{\substack{
h\ \mod{q}\\
h\not\equiv0\ \mod{q}
}}
\biggl|S\biggl(2N+1;\frac{h}{q}+\frac{\ell}{g^{3}-g}\biggr)\biggr|
\end{align}
since the contribution of the term with $q=1$ is zero.
We thus bound the sum
\[
S
\coloneqq
\sum_{\substack{
2\le q\le Q\\
(q,g^{3}-g)=1
}}
q^{-\frac{3}{2}}
\sum_{\substack{
h\ \mod{q}\\
h\not\equiv0\ \mod{q}
}}
\biggl|S\biggl(2N+1;\frac{h}{q}+\frac{\ell}{g^{3}-g}\biggr)\biggr|
\]
for arbitrary $\ell\in\mathbb{Z}$.
By classifying the terms by the value of $d=(h,q)$, we have
\begin{align}
S
&=
\sum_{\substack{
2\le q\le Q\\
(q,g^{3}-g)=1
}}
q^{-\frac{3}{2}}
\sum_{\substack{d\mid q\\d\neq q}}
\sum_{\substack{
1\le h<q\\
(h,q)=d
}}
\biggl|S\biggl(2N+1;\frac{h}{q}+\frac{\ell}{g^{3}-g}\biggr)\biggr|\\
&=
\sum_{\substack{
d\le Q/2\\
(d,g^{3}-g)=1
}}
\sum_{\substack{
2\le q\le Q\\
(q,g^{3}-g)=1\\
d\mid q\\
q\neq d
}}
q^{-\frac{3}{2}}
\sum_{\substack{
1\le h<q\\
(h,q)=d
}}
\biggl|S\biggl(2N+1;\frac{h}{q}+\frac{\ell}{g^{3}-g}\biggr)\biggr|\\
&=
\sum_{\substack{
d\le Q/2\\
(d,g^{3}-g)=1
}}
d^{-\frac{3}{2}}
\sum_{\substack{
2\le q\le Q/d\\
(q,g^{3}-g)=1
}}
q^{-\frac{3}{2}}
\astsum_{h\ \mod{q}}
\biggl|S\biggl(2N+1;\frac{h}{q}+\frac{\ell}{g^{3}-g}\biggr)\biggr|.
\end{align}
We decompose the inner sum dyadically and reduce its estimate to the dyadic sum
\[
T_{R}
\coloneqq
\sum_{\substack{
R/2\le q\le R\\
(q,g^{3}-g)=1
}}
q^{-\frac{3}{2}}
\astsum_{h\ \mod{q}}
\biggl|S\biggl(2N+1;\frac{h}{q}+\frac{\ell}{g^{3}-g}\biggr)\biggr|
\]
with $2\le R\le Q/d$. By \cref{lem:S_to_Phi} and \cref{lem:TuxanidyPanario_hybrid}
with $\delta=\frac{2}{5}-\frac{\epsilon}{2}\in[\frac{1}{3},\frac{2}{5}-\frac{\epsilon}{2}]\subset(0,1]$, we have
\begin{align}
T_{R}
&\ll
R^{-\frac{3}{2}}
\sum_{\substack{
2\le q\le R\\
(q,g^{3}-g)=1
}}
\astsum_{h\ \mod{q}}
\biggl|S\biggl(2N+1;\frac{h}{q}+\frac{\ell}{g^{3}-g}\biggr)\biggr|\\
&\ll
R^{-\frac{3}{2}}
\sum_{\substack{
2\le q\le R\\
(q,g^{3}-g)=1
}}
\astsum_{h\ \mod{q}}
\Phi_{N}\biggl(\frac{h}{q}+\frac{\ell}{g^{3}-g}\biggr)\\
&\ll
(
R^{\frac{1}{2}}g^{(1-(\frac{2}{5}-\frac{\epsilon}{2})-\sigma_{1}(g,\frac{\epsilon}{2}))N}
+
R^{-\frac{1}{2}-\sigma_{1}(g,\frac{\epsilon}{2})}
g^{N}
)
e^{-\frac{\sigma_{\infty}(g)N}{\log R}}.
\end{align}
By summing over $R$ dyadically with noting that
\[
R^{-\frac{1}{4}-\sigma_{1}(g,\frac{\epsilon}{2})}
e^{-\frac{\sigma_{\infty}(g)N}{\log R}}
\le
g^{-c\sqrt{N}}
\]
for $R\ge2$ with some $c=c(g,\epsilon)>0$, since $Q\le g^{(\frac{4}{5}-\epsilon)N}$, we get
\begin{align}
\sum_{\substack{
2\le q\le Q/d\\
(q,g^{3}-g)=1
}}
q^{-\frac{3}{2}}
\astsum_{h\ \mod{q}}
\biggl|S\biggl(2N+1;\frac{h}{q}+\frac{\ell}{g^{3}-g}\biggr)\biggr|
&\ll
(Q/d)^{\frac{1}{2}}g^{(1-(\frac{2}{5}-\frac{\epsilon}{2})-\sigma_{1}(g,\frac{\epsilon}{2}))N}
+
g^{N-c\sqrt{N}}\\
&\ll
g^{(1-\sigma_{1}(g,\frac{\epsilon}{2}))N}
+
g^{N-c\sqrt{N}}
\ll
g^{N-c\sqrt{N}}.
\end{align}
On inserting this estimate into the above estimate of $S$,
we obtain the lemma.
\end{proof}

We also need the distribution of palindromes
divisible by a given divisor $d$ of $g^{3}-g$,
which seems not to be discussed so much in the preceding studies:
\begin{lemma}
\label{lem:palindrome_AP_g3_g}
For $N\in\mathbb{N}$ and $d\mid g^{3}-g$, we have
\[
\#\Pi(2N+1;0,d)
=
\biggl(1+\frac{(d,g^{2}-1,2)}{g-1}\biggr)
\biggl(1-\frac{(d,g)}{g}\biggr)
\frac{1}{d}\#\Pi(2N+1)
+O(1),
\]
where the implicit constant depends only on $g$.
\end{lemma}
\begin{proof}
For $n\in\Pi(2N+1)$, since $(g,g^{2}-1)=1$ and $d=(d,g^{3}-g)=(d,g)(d,g^{2}-1)$, we have
\begin{equation}
\label{lem:palindrome_AP_g3_g:divisibility_rephrase}
d\mid n
\iff
\left\{
\begin{aligned}
n&\equiv0\ \mod{(d,g)},\\
n&\equiv0\ \mod{(d,g^{2}-1)}.
\end{aligned}
\right.
\end{equation}
We then write $n\in\Pi(2N+1)$ as
\begin{equation}
\label{lem:palindrome_AP_g3_g:n_expansion}
n
=
\sum_{i=0}^{N-1}
n_{i}(g^{i}+g^{2N-i})
+
n_{N}g^{N}
\quad\text{with}\quad
n_{0},\ldots,n_{N}\in\{0,\ldots,g-1\}
\ \text{and}\ 
n_{0}\neq0.
\end{equation}
Under this expansion \cref{lem:palindrome_AP_g3_g:n_expansion}, we have
\[
n
\equiv
n_{0}
\ \mod{(d,g)}
\]
and so the first congruence on the right-hand side of \cref{lem:palindrome_AP_g3_g:divisibility_rephrase}
can be rephrased as
\begin{equation}
\label{lem:palindrome_AP_g3_g:d_g_cond_rephrased}
n\equiv0\ \mod{(d,g)}
\iff
(d,g)\mid n_{0}.
\end{equation}
Also, under the expansion \cref{lem:palindrome_AP_g3_g:n_expansion}, we have
\begin{align}
n
\equiv
2
\sum_{i=0}^{N-1}
n_{i}g^{i}
+
n_{N}g^{N}\ \mod{(d,g^{2}-1)}
\end{align}
since
\[
g^{2N-i}
=
(g^{2})^{N-i}g^{i}
\equiv
g^{i}\ \mod{(d,g^{2}-1)}
\quad\text{for $i\in\{0,\ldots,N-1\}$}.
\]
Thus, the second congruence on the right-hand side of \cref{lem:palindrome_AP_g3_g:divisibility_rephrase}
can be rephrased as
\begin{equation}
\label{lem:palindrome_AP_g3_g:d_g2_1_cond_rephrased_pre}
n\equiv0\ \mod{(d,g^{2}-1)}
\iff
2
\sum_{i=1}^{N-1}
n_{i}g^{i}
\equiv
-2n_{0}-n_{N}g^{N}
\ \mod{(d,g^{2}-1)}.
\end{equation}
The right-hand side of this equivalence implies
\begin{equation}
\label{lem:palindrome_AP_g3_g:d_g2_1_cond_nN_cond}
(d,g^{2}-1,2)\mid n_{N}.
\end{equation}
Thus, the condition \cref{lem:palindrome_AP_g3_g:d_g2_1_cond_rephrased_pre} is further rephrased as
\begin{equation}
\label{lem:palindrome_AP_g3_g:d_g2_1_cond_rephrased}
\begin{aligned}
n\equiv0\ \mod{(d,g^{2}-1)}
&\iff
\left\{
\begin{gathered}
2
\sum_{i=1}^{N-1}
n_{i}g^{i}
\equiv
-2n_{0}-n_{N}g^{N}
\ \mod{(d,g^{2}-1)},\\
(d,g^{2}-1,2)\mid n_{N}
\end{gathered}
\right.\\
&\iff
\left\{
\begin{gathered}
\sum_{i=1}^{N-1}
n_{i}g^{i-1}
\equiv
a(n_{0},n_{N},g)
\ \mod{\frac{(d,g^{2}-1)}{(d,g^{2}-1,2)}},\\
(d,g^{2}-1,2)\mid n_{N}
\end{gathered}
\right.
\end{aligned}
\end{equation}
with the residue $a(n_{0},n_{N},g)\ \mod{\frac{(d,g^{2}-1)}{(d,g^{2}-1,2)}}$ determined by
\[
2ga(n_{0},n_{N},g)
\equiv
-2n_{0}-n_{N}g^{N}
\ \mod{(d,g^{2}-1)}
\]
provided $(d,g^{2}-1,2)\mid n_{N}$.
By combining \cref{lem:palindrome_AP_g3_g:d_g_cond_rephrased} and \cref{lem:palindrome_AP_g3_g:d_g2_1_cond_rephrased}
we get
\begin{equation}
\label{lem:palindrome_AP_g3_g:first_reduction}
\#\Pi(2N+1,0,d)
=
\sum_{\substack{
1\le n_{0}<g\\
(d,g)\mid n_{0}
}}
\sum_{\substack{
0\le n_{N}<g\\
(d,g^{2}-1,2)\mid n_{N}
}}
\sum_{\substack{
0\le n_{1},\ldots,n_{N-1}<g\\
\sum_{i=1}^{N-1}
n_{i}g^{i-1}
\equiv
a(n_{0},n_{N},g)
\ \mod{\frac{(d,g^{2}-1)}{(d,g^{2}-1,2)}}
}}
1.
\end{equation}
By considering the base $g$ expansion, we find that the form
\[
\sum_{i=1}^{N-1}
n_{i}g^{i-1}
\quad\text{with}\quad
0\le n_{1},\ldots,n_{N-1}<g
\]
runs over the integers in $[0,g^{N-1})$ one by one.
We thus have
\begin{equation}
\label{lem:palindrome_AP_g3_g:after_folding_base_g_expansion}
\begin{aligned}
\#\Pi(2N+1,0,d)
&=
\sum_{\substack{
1\le n_{0}<g\\
(d,g)\mid n_{0}
}}
\sum_{\substack{
0\le n_{N}<g\\
(d,g^{2}-1,2)\mid n_{N}
}}
\sum_{\substack{
0\le n<g^{N-1}\\
n
\equiv
a(n_{0},n_{N},g)
\ \mod{\frac{(d,g^{2}-1)}{(d,g^{2}-1,2)}}
}}
1\\
&=
\frac{(d,g^{2}-1,2)}{(d,g^{2}-1)}g^{N-1}
\sum_{\substack{
1\le n_{0}<g\\
(d,g)\mid n_{0}
}}
\sum_{\substack{
0\le n_{N}<g\\
(d,g^{2}-1,2)\mid n_{N}
}}
1
+
O(1).
\end{aligned}
\end{equation}
We have
\begin{equation}
\label{lem:palindrome_AP_g3_g:n0_counting}
\sum_{\substack{
1\le n_{0}<g\\
(d,g)\mid n_{0}
}}
1
=
\sum_{\substack{
1\le n_{0}\le g\\
(d,g)\mid n_{0}
}}
1
-1
=
\frac{g}{(d,g)}
\biggl(1-\frac{(d,g)}{g}\biggr).
\end{equation}
Also, by noting that $(d,g^{2}-1,2)=2$ occurs only when $g$ is odd, we have
\begin{equation}
\label{lem:palindrome_AP_g3_g:nN_counting}
\begin{aligned}
\sum_{\substack{
0\le n_{N}<g\\
(d,g^{2}-1,2)\mid n_{N}
}}
1
&=
\left\{
\begin{array}{>{\displaystyle}cl}
g&\text{if $(d,g^{2}-1,2)=1$},\\[2mm]
\frac{g+1}{2}&\text{if $(d,g^{2}-1,2)=2$}
\end{array}
\right.\\
&=
\frac{g-1}{(d,g^{2}-1,2)}
\biggl(1+\frac{(d,g^{2}-1,2)}{g-1}\biggr).
\end{aligned}
\end{equation}
On inserting \cref{lem:palindrome_AP_g3_g:n0_counting}
and \cref{lem:palindrome_AP_g3_g:nN_counting}
into \cref{lem:palindrome_AP_g3_g:after_folding_base_g_expansion}
and recalling \cref{lem:Pi_count},
we obtain the assertion.
\end{proof}

For \cref{thm:relatively_prime_palindrome:g3_g},
we need the following results:
\begin{lemma}[{Tuxanidy--Panario~\cite{TuxanidyPanario}}]
\label{lem:TuxanidyPanario_BV}
For $x,Q\ge1$, $\epsilon>0$, we have
\[
\sum_{\substack{
q\le Q\\
(q,g^{3}-g)=1
}}
\sup_{y\le x}
\max_{a\in\mathbb{Z}}
\biggl|
\#\mathscr{P}^{\ast}(y;a,q)
-
\frac{1}{q}\#\mathscr{P}^{\ast}(y)
\biggr|
\ll
\#\mathscr{P}^{\ast}(x)
e^{-c\sqrt{\log x}}
\]
with some $c>0$ provided
\[
Q\le x^{\frac{1}{5}-\epsilon},
\]
where the constant $c$ and the implicit constant depends only on $g$ and $\epsilon$.
\end{lemma}
\begin{proof}
See Theorem~1.5 of \cite{TuxanidyPanario}.
\end{proof}

\begin{lemma}[{Tuxanidy--Panario~\cite{TuxanidyPanario}}]
\label{lem:P_ast_size}
For $x\ge1$, we have
\[
\#\mathscr{P}^{\ast}(x)\asymp\sqrt{x},
\]
where the implicit constant depends only on $g$.
\end{lemma}
\begin{proof}
See the last assertion of Lemma~9.1
of \cite{TuxanidyPanario}.
\end{proof}

\section{Proof of the main theorems}
\label{sec:proof_mainthm}
We now prove \cref{thm:relatively_prime_palindrome}
and \cref{thm:relatively_prime_palindrome:g3_g}.
\begin{proof}[Proof of \cref{thm:relatively_prime_palindrome}.]
Let us write
\[
\mathscr{N}
\coloneqq
\sum_{\substack{
m,n\in\Pi(2N+1)\\
(m,n)=1
}}
1.
\]
By using the well-known formula
\[
\sum_{d\mid M}\mu(d)
=
\mathbbm{1}_{M=1}
\quad\text{for $M\in\mathbb{N}$},
\]
we have
\begin{align}
\mathscr{N}
=
\sum_{m,n\in\Pi(2N+1)}
\sum_{d\mid(m,n)}
\mu(d)
=
\sum_{d<g^{2N+1}}
\mu(d)
\#\Pi(2N+1;0,d)^{2}.
\end{align}
Let $U\coloneqq g^{\frac{1}{5}N}$ and dissect the above sum as
\begin{equation}
\label{thm:relatively_prime_palindrome:N_N1_N2}
\mathscr{N}
=
\sum_{d\le U}
\mu(d)
\#\Pi(2N+1;0,d)^{2}
+
\sum_{U<d<g^{2N+1}}
\mu(d)
\#\Pi(2N+1;0,d)^{2}
\eqqcolon
\mathscr{N}_{1}+\mathscr{N}_{2}.
\end{equation}
We then calculate $\mathscr{N}_{1}$ and $\mathscr{N}_{2}$ separately.

For the sum $\mathscr{N}_{1}$, we first write the squarefree number $d$ as
\[
d=mf
\quad\text{with}\quad
(m,g^{3}-g)=1
\ \text{and}\ 
f=(d,g^{3}-g).
\]
Note that then $(m,f)=1$. This gives
\begin{equation}
\label{thm:relatively_prime_palindrome:N1:d_ef_decomp}
\mathscr{N}_{1}
=
\sum_{f\mid g^{3}-g}
\mu(f)
\sum_{\substack{
m\le U/f\\
(m,g^{3}-g)=1
}}
\mu(m)
\#\Pi(2N+1;0,mf)^{2}.
\end{equation}
For $\#\Pi(2N+1;0,mf)$, we first classify palindromes $\mod{g^{3}-g}$ to get
\begin{equation}
\label{thm:relatively_prime_palindrome:N1:g3_g_decomp}
\#\Pi(2N+1;0,mf)
=
\sum_{\substack{
1\le k\le g^{3}-g\\
f\mid k
}}
\#\Pi_{k,g}(2N+1;0,m).
\end{equation}
We compare this quantity with
\begin{equation}
\label{thm:relatively_prime_palindrome:N1:g3_g_decomp_counterpart}
\frac{1}{m}
\#\Pi(2N+1;0,f)
=
\sum_{\substack{
1\le k\le g^{3}-g\\
f\mid k
}}
\frac{1}{m}
\#\Pi_{k,g}(2N+1).
\end{equation}
We then have
\begin{align}
&\#\Pi(2N+1;0,mf)^{2}
-
\biggl(
\frac{1}{m}
\#\Pi(2N+1;0,f)
\biggr)^{2}\\
&=
\biggl(
\#\Pi(2N+1;0,mf)
-
\frac{1}{m}
\#\Pi(2N+1;0,f)
\biggr)
\biggl(
\#\Pi(2N+1;0,mf)
+
\frac{1}{m}
\#\Pi(2N+1;0,f)
\biggr)\\
&=
\biggl(
\sum_{\substack{
1\le k\le g^{3}-g\\
f\mid k
}}
\Bigl(
\#\Pi_{k,g}(2N+1;0,m)
-
\frac{1}{m}
\#\Pi_{k,g}(2N+1)\Bigr)\biggr)\\
&\hspace{40mm}
\times\biggl(
\#\Pi(2N+1;0,mf)
+
\frac{1}{m}
\#\Pi(2N+1;0,f)
\biggr).
\end{align}
By \cref{lem:BS_Brun_Titchmarsh}, we have
\[
\#\Pi(2N+1;0,mf)
+
\frac{1}{m}
\#\Pi(2N+1;0,f)
\ll
g^{N}m^{-\frac{1}{2}}
\]
and so
\begin{align}
&\#\Pi(2N+1;0,mf)^{2}
-
\biggl(
\frac{1}{m}
\#\Pi(2N+1;0,f)
\biggr)^{2}\\
&\ll
g^{N}m^{-\frac{1}{2}}
\sum_{\substack{
1\le k\le g^{3}-g\\
f\mid k
}}
\Bigl|
\#\Pi_{k,g}(2N+1;0,m)
-
\frac{1}{m}
\#\Pi_{k,g}(2N+1)
\Bigr|.
\end{align}
On inserting this into \cref{thm:relatively_prime_palindrome:N1:d_ef_decomp},
we get
\begin{equation}
\label{thm:relatively_prime_palindrome:N1:M1_E1}
\mathscr{N}_{1}
=
M_{1}+O(E_{1}),
\end{equation}
where
\begin{align}
M_{1}
&\coloneqq
\sum_{f\mid g^{3}-g}
\mu(f)
\#\Pi(2N+1;0,f)^{2}
\sum_{\substack{
m\le U/f\\
(m,g^{3}-g)=1
}}
\frac{\mu(m)}{m^{2}},\\
E_{1}
&\coloneqq
g^{N}
\sum_{f\mid g^{3}-g}
\sum_{\substack{
1\le k\le g^{3}-g\\
f\mid k
}}
\sum_{\substack{
m\le U/f\\
(m,g^{3}-g)=1
}}
m^{-\frac{1}{2}}
\Bigl|
\#\Pi_{k,g}(2N+1;0,m)
-
\frac{1}{m}
\#\Pi_{k,g}(2N+1)
\Bigr|.
\end{align}
For $M_{1}$, since
\[
\sum_{\substack{
m\le U/f\\
(m,g^{3}-g)=1
}}
\frac{\mu(m)}{m^{2}}
=
\frac{1}{\zeta(2)}
\biggl(\prod_{p\mid g^{3}-g}\biggl(1-\frac{1}{p^{2}}\biggr)\biggr)^{-1}
+
O\biggl(\frac{f}{U}\biggr),
\]
by \cref{lem:Pi_count}, we have
\begin{equation}
\label{thm:relatively_prime_palindrome:N1:M1:after_inner_sum}
\begin{aligned}
M_{1}
&=
\frac{1}{\zeta(2)}
\biggl(\prod_{p\mid g^{3}-g}\biggl(1-\frac{1}{p^{2}}\biggr)\biggr)^{-1}
\sum_{f\mid g^{3}-g}
\mu(f)
\#\Pi(2N+1;0,f)^{2}
+
O(g^{2N}U^{-1})\\
&\eqqcolon
M_{11}
+
O(g^{2N}U^{-1}),\quad\text{say}.
\end{aligned}
\end{equation}
On inserting \cref{lem:palindrome_AP_g3_g}, we have
\begin{equation}
\label{thm:relatively_prime_palindrome:N1:M11:after_inner_sum}
\begin{aligned}
M_{11}
&=
\frac{1}{\zeta(2)}
\#\Pi(2N+1)^{2}
\biggl(\prod_{p\mid g^{3}-g}\biggl(1-\frac{1}{p^{2}}\biggr)\biggr)^{-1}\\
&\hspace{0.1\textwidth}
\times
\sum_{f\mid g^{3}-g}
\frac{\mu(f)}{f^{2}}
\biggl(1+\frac{(f,g^{2}-1,2)}{g-1}\biggr)^{2}
\biggl(1-\frac{(f,g)}{g}\biggr)^{2}
+
O(g^{N}).
\end{aligned}
\end{equation}
When $g$ is even, we have
\begin{equation}
\label{thm:relatively_prime_palindrome:N1:M11:after_inner_sum_even}
\begin{aligned}
M_{11}
&=
\frac{1}{\zeta(2)}
\#\Pi(2N+1)^{2}
\biggl(\frac{g}{g-1}\biggr)^{2}
\biggl(\prod_{p\mid g^{3}-g}\biggl(1-\frac{1}{p^{2}}\biggr)\biggr)^{-1}\\
&\hspace{0.1\textwidth}
\times
\sum_{f\mid g^{3}-g}
\frac{\mu(f)}{f^{2}}
\biggl(1-\frac{2(f,g)}{g}+\frac{(f,g)^{2}}{g^{2}}\biggr)
+
O(g^{N}).
\end{aligned}
\end{equation}
By the multiplicativity of $(f,g)$ as a function of $f$ and by $(g,g^{2}-1)=1$, we have
\begin{equation}
\label{thm:relatively_prime_palindrome:N1:M11:f_product}
\begin{aligned}
&\sum_{f\mid g^{3}-g}
\frac{\mu(f)}{f^{2}}
\biggl(1-\frac{2(f,g)}{g}+\frac{(f,g)^{2}}{g^{2}}\biggr)\\
&=
\prod_{p\mid g^{3}-g}\biggl(1-\frac{1}{p^{2}}\biggr)
-
\frac{2}{g}
\prod_{p\mid g}\biggl(1-\frac{1}{p}\biggr)
\prod_{p\mid g^{2}-1}\biggl(1-\frac{1}{p^{2}}\biggr)
+
\frac{1}{g^{2}}
\prod_{p\mid g}(1-1)
\prod_{p\mid g^{2}-1}\biggl(1-\frac{1}{p^{2}}\biggr)\\
&=
\prod_{p\mid g^{3}-g}\biggl(1-\frac{1}{p^{2}}\biggr)
\biggl(
1-\frac{2}{g}\prod_{p\mid g}\biggl(\frac{p}{p+1}\biggr)
\biggr).
\end{aligned}
\end{equation}
Therefore, by \cref{thm:relatively_prime_palindrome:N1:M11:after_inner_sum_even}, when $g$ is even, we have
\begin{equation}
\label{thm:relatively_prime_palindrome:N1:M11:g_even}
M_{11}
=
\frac{1}{\zeta(2)}
\biggl(\frac{g}{g-1}\biggr)^{2}
\biggl(
1-\frac{2}{g}\prod_{p\mid g}\biggl(\frac{p}{p+1}\biggr)
\biggr)
\#\Pi(2N+1)^{2}
+
O(g^{N}).
\end{equation}
When $g$ is odd,
by \cref{thm:relatively_prime_palindrome:N1:M1:after_inner_sum}
and \cref{lem:palindrome_AP_g3_g},
we have
\begin{equation}
\label{thm:relatively_prime_palindrome:N1:M11:after_inner_sum_odd}
\begin{aligned}
M_{11}
&=
\frac{1}{\zeta(2)}
\#\Pi(2N+1)^{2}
\biggl(\frac{g+1}{g-1}\biggr)^{2}
\biggl(\prod_{p\mid g^{3}-g}\biggl(1-\frac{1}{p^{2}}\biggr)\biggr)^{-1}\\
&\hspace{0.1\textwidth}
\times
\sum_{f\mid g^{3}-g}
\frac{\mu(f)}{f^{2}}
\biggl(1-\frac{2(f,g)}{g}+\frac{(f,g)^{2}}{g^{2}}\biggr)\\
&\hspace{0.05\textwidth}
-
\frac{1}{\zeta(2)}
\#\Pi(2N+1)^{2}
\frac{2g+1}{(g-1)^{2}}
\biggl(\prod_{p\mid g^{3}-g}\biggl(1-\frac{1}{p^{2}}\biggr)\biggr)^{-1}\\
&\hspace{0.15\textwidth}
\times
\sum_{\substack{
f\mid g^{3}-g\\
f:\text{odd}
}}
\frac{\mu(f)}{f^{2}}
\biggl(1-\frac{2(f,g)}{g}+\frac{(f,g)^{2}}{g^{2}}\biggr)
+O(g^{N}).
\end{aligned}
\end{equation}
By \cref{thm:relatively_prime_palindrome:N1:M11:f_product} and its variant
\begin{equation}
\begin{aligned}
&\sum_{\substack{
f\mid g^{3}-g\\
f:\text{odd}
}}
\frac{\mu(f)}{f^{2}}
\biggl(1-\frac{2(f,g)}{g}+\frac{(f,g)^{2}}{g^{2}}\biggr)\\
&=
\prod_{\substack{
p\mid g^{3}-g\\
p>2
}}\biggl(1-\frac{1}{p^{2}}\biggr)
-
\frac{2}{g}
\prod_{\substack{
p\mid g\\
p>2
}}
\biggl(1-\frac{1}{p}\biggr)
\prod_{\substack{
p\mid g^{2}-1\\
p>2
}}
\biggl(1-\frac{1}{p^{2}}\biggr)
+
\frac{1}{g^{2}}
\prod_{\substack{
p\mid g\\
p>2
}}
(1-1)
\prod_{\substack{
p\mid g^{2}-1\\
p>2
}}
\biggl(1-\frac{1}{p^{2}}\biggr)\\
&=
\frac{4}{3}
\prod_{p\mid g^{3}-g}\biggl(1-\frac{1}{p^{2}}\biggr)
\biggl(
1-\frac{2}{g}\prod_{p\mid g}\biggl(\frac{p}{p+1}\biggr)
\biggr),
\end{aligned}
\end{equation}
we can calculate \cref{thm:relatively_prime_palindrome:N1:M11:after_inner_sum_odd} as
\begin{equation}
\label{thm:relatively_prime_palindrome:N1:M11:g_odd}
M_{11}
=
\frac{1}{\zeta(2)}
\frac{g+\frac{1}{3}}{g-1}
\biggl(
1-\frac{2}{g}\prod_{p\mid g}\biggl(\frac{p}{p+1}\biggr)
\biggr)
\#\Pi(2N+1)^{2}
+
O(g^{N}).
\end{equation}
By recalling \cref{thm:relatively_prime_palindrome:rho_def}
and inserting
\cref{thm:relatively_prime_palindrome:N1:M11:g_even}
and \cref{thm:relatively_prime_palindrome:N1:M11:g_odd}
into \cref{thm:relatively_prime_palindrome:N1:M1:after_inner_sum},
we have
\begin{equation}
\label{thm:relatively_prime_palindrome:N1:M1}
M_{1}
=
\frac{\rho(g)}{\zeta(2)}
\biggl(
1-\frac{2}{g}\prod_{p\mid g}\biggl(\frac{p}{p+1}\biggr)
\biggr)
\#\Pi(2N+1)^{2}
+
O(g^{2N}U^{-1}+g^{N}).
\end{equation}
For $E_{1}$, since $U=g^{(\frac{4}{5}-\frac{3}{5})N}$,
\cref{lem:Tuxanidy_Panario_BV_variant}, we have
\begin{equation}
\label{thm:relatively_prime_palindrome:N1:E1}
E_{1}
\ll
g^{N}
\sum_{\substack{
m\le U\\
(m,g^{3}-g)=1
}}
m^{-\frac{1}{2}}
\max_{a,k\in\mathbb{Z}}
\Bigl|
\#\Pi_{k,g}(2N+1;a,m)
-
\frac{1}{m}
\#\Pi_{k,g}(2N+1)
\Bigr|
\ll
g^{2N-c\sqrt{N}}.
\end{equation}
On inserting \cref{thm:relatively_prime_palindrome:N1:M1}
and \cref{thm:relatively_prime_palindrome:N1:E1}
into \cref{thm:relatively_prime_palindrome:N1:M1_E1}
with recalling $U=g^{\frac{1}{5}N}$,
we arrive at
\begin{equation}
\label{thm:relatively_prime_palindrome:N1}
\mathscr{N}_{1}
=
\frac{\rho(g)}{\zeta(2)}
\biggl(
1-\frac{2}{g}\prod_{p\mid g}\biggl(\frac{p}{p+1}\biggr)
\biggr)
\#\Pi(2N+1)^{2}
+
O(g^{2N-c\sqrt{N}}).
\end{equation}

For the sum $\mathscr{N}_{2}$, we apply \cref{lem:BS_Brun_Titchmarsh}
to only one factor $\#\Pi(2N+1;0,d)$ to get
\begin{align}
\mathscr{N}_{2}
\ll
g^{N}U^{-\frac{1}{2}}
\sum_{U<d<g^{2N+1}}
\sum_{\substack{
n\in\Pi(2N+1)\\
d\mid n
}}
1
=
g^{N}U^{-\frac{1}{2}}
\sum_{n\in\Pi(2N+1)}
\sum_{\substack{
d\mid n\\
U<d<g^{2N+1}
}}
1
\le
g^{N}U^{-\frac{1}{2}}
\sum_{n\in\Pi(2N+1)}
\tau(n).
\end{align}
By using the bound $\tau(n)\ll_{\epsilon}n^{\epsilon}$
and \cref{lem:Pi_count}, we have
\begin{equation}
\label{thm:relatively_prime_palindrome:N2}
\mathscr{N}_{2}
\ll_{\epsilon}
g^{(2+\epsilon)N}U^{-\frac{1}{2}}
\ll
g^{2N-c\sqrt{N}}
\end{equation}
by taking $\epsilon=\frac{1}{20}$ since $U=g^{\frac{1}{5}N}$.

On inserting \cref{thm:relatively_prime_palindrome:N1} and \cref{thm:relatively_prime_palindrome:N2}
into \cref{thm:relatively_prime_palindrome:N_N1_N2}, we get
\begin{align}
\mathscr{N}
=
\frac{\rho(g)}{\zeta(2)}
\biggl(
1-\frac{2}{g}\prod_{p\mid g}\biggl(\frac{p}{p+1}\biggr)
\biggr)
\#\Pi(2N+1)^{2}
+
O(g^{2N-c\sqrt{N}}).
\end{align}
On recalling \cref{lem:Pi_count},
we obtain the theorem.
\end{proof}

\begin{proof}[Proof of \cref{thm:relatively_prime_palindrome:g3_g}.]
We use the same argument as the proof of \cref{thm:relatively_prime_palindrome}.
Let us write
\[
\mathscr{N}^{\ast}
\coloneqq
\sum_{\substack{
m,n\in\mathscr{P}^{\ast}(x)\\
(m,n)=1
}}
1.
\]
By using the M\"obius function, we have
\begin{align}
\mathscr{N}^{\ast}
=
\sum_{m,n\in\mathscr{P}^{\ast}(x)}
\sum_{d\mid(m,n)}
\mu(d)
=
\sum_{\substack{
d<x\\
(d,g^{3}-g)=1
}}
\mu(d)
\#\mathscr{P}^{\ast}(x;0,d)^{2}
\end{align}
since the elements of $\mathscr{P}^{\ast}(x)$
are all coprime to $g^{3}-g$.
Let $U\coloneqq x^{\frac{1}{10}}$ and dissect the above sum as
\begin{equation}
\label{thm:relatively_prime_palindrome:g3_g:N_N1_N2}
\mathscr{N}^{\ast}
=
\sum_{\substack{
d\le U\\
(d,g^{3}-g)=1
}}
\mu(d)
\#\mathscr{P}^{\ast}(x;0,d)^{2}
+
\sum_{\substack{
U<d\le x\\
(d,g^{3}-g)=1
}}
\mu(d)
\#\mathscr{P}^{\ast}(x;0,d)^{2}
\eqqcolon
\mathscr{N}_{1}^{\ast}+\mathscr{N}_{2}^{\ast}.
\end{equation}
We then calculate $\mathscr{N}_{1}^{\ast}$ and $\mathscr{N}_{2}^{\ast}$ separately.

For the sum $\mathscr{N}_{1}^{\ast}$, we have
\begin{align}
\#\mathscr{P}^{\ast}(x;0,d)^{2}
-
\biggl(\frac{1}{d}\#\mathscr{P}^{\ast}(x)\biggr)^{2}
&=
\biggl(
\#\mathscr{P}^{\ast}(x;0,d)
-
\frac{1}{d}
\#\mathscr{P}^{\ast}(x)
\biggr)
\biggl(
\#\mathscr{P}^{\ast}(x;0,d)
+
\frac{1}{d}
\#\mathscr{P}^{\ast}(x)
\biggr)\\
&\ll
\#\mathscr{P}^{\ast}(x)
\biggl|
\#\mathscr{P}^{\ast}(x;0,d)
-
\frac{1}{d}
\#\mathscr{P}^{\ast}(x)
\biggr|.
\end{align}
We thus have
\begin{equation}
\label{thm:relatively_prime_palindrome:g3_g:N1:M1_E1}
\mathscr{N}_{1}^{\ast}
=
M_{1}^{\ast}
+
O(E_{1}^{\ast}),
\end{equation}
where
\[
M_{1}^{\ast}
\coloneqq
\#\mathscr{P}^{\ast}(x)^{2}
\sum_{\substack{
d\le U\\
(d,g^{3}-g)=1
}}
\frac{\mu(d)}{d^{2}}
\and
E_{1}^{\ast}
\coloneqq
\#\mathscr{P}^{\ast}(x)
\sum_{\substack{
d\le U\\
(d,g^{3}-g)=1
}}
\biggl|
\#\mathscr{P}^{\ast}(x;0,d)
-
\frac{1}{d}
\#\mathscr{P}^{\ast}(x)
\biggr|.
\]
We clearly have
\begin{equation}
\label{thm:relatively_prime_palindrome:g3_g:N1:M1}
M_{1}^{\ast}
=
\frac{1}{\zeta(2)}
\biggl(\prod_{p\mid g^{3}-g}\biggl(1-\frac{1}{p^{2}}\biggr)\biggr)^{-1}
\#\mathscr{P}^{\ast}(x)^{2}
+
O(\#\mathscr{P}^{\ast}(x)^{2}U^{-1}).
\end{equation}
Also, since $U=x^{\frac{1}{5}-\frac{1}{10}}$, \cref{lem:TuxanidyPanario_BV} implies
\begin{equation}
\label{thm:relatively_prime_palindrome:g3_g:N1:E1}
E_{1}^{\ast}
\ll
\#\mathscr{P}^{\ast}(x)^{2}
e^{-c\sqrt{\log x}}
\end{equation}
with some $c>0$. On inserting \cref{thm:relatively_prime_palindrome:g3_g:N1:M1}
and \cref{thm:relatively_prime_palindrome:g3_g:N1:E1} into
\cref{thm:relatively_prime_palindrome:g3_g:N1:M1_E1}, we arrive at
\begin{equation}
\label{thm:relatively_prime_palindrome:g3_g:N1}
\mathscr{N}_{1}^{\ast}
=
\frac{1}{\zeta(2)}
\biggl(\prod_{p\mid g^{3}-g}\biggl(1-\frac{1}{p^{2}}\biggr)\biggr)^{-1}
\#\mathscr{P}^{\ast}(x)^{2}
+
O(
\#\mathscr{P}^{\ast}(x)^{2}e^{-c\sqrt{\log x}}
)
\end{equation}
since $U=x^{\frac{1}{10}}$.

We next estimate $\mathscr{N}_{2}^{\ast}$.
By dissecting dyadically and using \cref{lem:BS_Brun_Titchmarsh},
we get
\[
\#\mathscr{P}^{\ast}(x;0,d)
\ll
\sum_{\substack{
N\ge 1\\
g^{N-1}\le x
}}
\#\Pi(N;0,d)
\ll
d^{-\frac{1}{2}}
\sum_{\substack{
N\ge 1\\
g^{N-1}\le x
}}
g^{\frac{N}{2}}
\ll
x^{\frac{1}{2}}d^{-\frac{1}{2}}.
\]
We apply this bound to only one factor $\#\mathscr{P}^{\ast}(x;0,d)$
and swapping the sum to get
\begin{align}
\mathscr{N}_{2}^{\ast}
\ll
x^{\frac{1}{2}}U^{-\frac{1}{2}}
\sum_{U<d<x}
\sum_{\substack{
n\in\mathscr{P}^{\ast}(x)\\
d\mid n
}}
1
\le
x^{\frac{1}{2}}U^{-\frac{1}{2}}
\sum_{n\in\mathscr{P}^{\ast}(x)}
\tau(n).
\end{align}
By using the bound $\tau(n)\ll_{\epsilon}n^{\epsilon}$
and \cref{lem:P_ast_size}, we have
\begin{equation}
\label{thm:relatively_prime_palindrome:g3_g:N2}
\mathscr{N}_{2}^{\ast}
\ll_{\epsilon}
x^{1+\epsilon}U^{-\frac{1}{2}}
\ll
x^{\frac{39}{40}}
\end{equation}
by taking $\epsilon=\frac{1}{40}$ since $U=x^{\frac{1}{10}}$.

On inserting \cref{thm:relatively_prime_palindrome:g3_g:N1} and \cref{thm:relatively_prime_palindrome:g3_g:N2}
into \cref{thm:relatively_prime_palindrome:g3_g:N_N1_N2}, we get
\begin{align}
\mathscr{N}^{\ast}
=
\frac{1}{\zeta(2)}
\biggl(\prod_{p\mid g^{3}-g}\biggl(1-\frac{1}{p^{2}}\biggr)\biggr)^{-1}
\#\mathscr{P}^{\ast}(x)^{2}
+
O(
\#\mathscr{P}^{\ast}(x)^{2}e^{-c\sqrt{\log x}}
+
x^{\frac{39}{40}}
).
\end{align}
On recalling \cref{lem:P_ast_size}, we obtain the theorem.
\end{proof}

\ifendnotes
\newpage
\begingroup
\parindent 0pt
\parskip 2ex
\def\enotesize{\normalsize}
\theendnotes
\endgroup
\fi

\bibliographystyle{amsplain}
\bibliography{RelativelyPrimePalindrome}
\bigskip

\begin{flushleft}
{\textsc{%
\small
Hirotaka Kobayashi\\[.3em]
\footnotesize
Graduate School of Mathematics, Nagoya University,\\
Furocho, Chikusa-ku, 464-8602 Nagoya, Japan.
}

\small
\textit{Email address}: \texttt{m17011z@math.nagoya-u.ac.jp}
}
\bigskip

{\textsc{%
\small
Yuta Suzuki\\[.3em]
\footnotesize
Department of Mathematics, Rikkyo University,\\
3-34-1 Nishi-Ikebukuro, Toshima-ku, Tokyo 171-8501, Japan.
}

\small
\textit{Email address}: \texttt{suzuyu@rikkyo.ac.jp}
}
\bigskip

{\textsc{%
\small
Ryota Umezawa\\[.3em]
\footnotesize
Graduate School of Mathematics, Nagoya University,\\
Furocho, Chikusa-ku, 464-8602 Nagoya, Japan.
}

\small
\textit{Email address}: \texttt{m15016w@math.nagoya-u.ac.jp}
}
\end{flushleft}

\end{document}